\documentclass[11pt]{article}

\usepackage[a4paper,margin=1in]{geometry}
\usepackage{amsmath,amssymb,amsthm}
\usepackage{graphicx}
\usepackage{booktabs}
\usepackage[T1]{fontenc}
\usepackage{lmodern}
\usepackage{microtype}

\newtheorem{theorem}{Theorem}
\newtheorem{remark}{Remark}

\title{ Second order mixed moment inequalities based on Gram matrices}
\author{Sergio Scarlatti\footnote{email: sergio.scarlatti@uniroma2.it}
\\
Dept. of Finance and Economics, University of Rome Tor-Vergata}\date{June 2026}

\begin{document}

\maketitle

\begin{abstract}
Recently [LT; Theorem 3.1] showed an extension of Walker's inequality [W] based on $N=3$ random variables. In this note we prove that extension is just a particular three-dimensional instance of a  general family  of second order mixed moment inequalities based on Gram matrices of arbitrary random vectors. We also discuss some implications of these inequalities on Cramer-Rao lower bound for biased estimators.
\end{abstract}

\section{Introduction and general settings}

There have been recent reconsiderations of a result of Walker [W] showing a self-improvement of Cauchy - Schwarz inequality, which is effective for non centered random variables, see [S] and [LT]. In particular [LT] presents a general inequality for $N=3$ r.v.'s which includes Walker's inequality as a special case. The aim of this note is to show that, by using Gram matrices,  a much more general result holds, and, furthermore, this generality can be used in statistical estimation theory to build improved bounds on the MSE of biased estimators. To this aim let $X=(X_1,\dots,X_N)$ be an $N$-dimensional random vector whose components belong to $L^2(\Omega,\mathcal F,\mathbb P)$, endowed with the inner product $\langle U,V\rangle=E(UV)$. By definition the Gram matrix of $X$ is $\text{Gram}(X)=G_N\equiv (E(X_hX_k))_{h,k=1}^N$. Hence $G_N=({g}_{hk})_{h,k=1}^N$ is symmetric positive semidefinite.

Fix two distinct indices $i\neq j$, we want to estimate the matrix entry ${g}_{ij}=E(X_iX_j)$ in terms of the remaining entries of $G_N$. To this aim, let $S=\{1,\ldots,N\}$ and define
\begin{equation}
S^{-}_{ij}\equiv \{1,\dots,N\}\setminus\{i,j\},
\end{equation}
moreover let $G_{S^{-}_{ij}}$ be the principal Gram block indexed by $S^{-}_{ij}$, namely
\begin{equation}
G_{S^{-}_{ij}}\equiv ({g}_{hk})_{h,k\in S^{-}_{ij}}.
\end{equation}
Throughout this note, we assume that $G_{S^{-}_{ij}}$ is positive definite, equivalently, the family $(X_h)_{h\in S^{-}_{ij}}$ is linearly independent in $L^2$.
Corresponding to a fixed index define the vector $g_i\equiv ({g}_{ih})_{h\in S^{-}_{ij}}$ , then set
\begin{equation}
\pi_{ij\mid S^{-}_{ij}}\equiv g_i^TG_{S^{-}_{ij}}^{-1}g_j,
\end{equation}
and
\begin{equation}
\Delta_{i\mid S^{-}_{ij}}\equiv {g}_{ii}-g_i^TG_{S^{-}_{ij}}^{-1}g_i,
\qquad
\Delta_{j\mid S^{-}_{ij}}\equiv {g}_{jj}-g_j^TG_{S^{-}_{ij}}^{-1}g_j.
\end{equation}
The quantity $\pi_{ij\mid S^{-}_{ij}}$ is the Gram inner product of the projections of $X_i$ and $X_j$ onto $\operatorname{span}\{X_h:h\in S^{-}_{ij}\}$, while $\Delta_{i\mid S^{-}_{ij}}$ and $\Delta_{j\mid S^{-}_{ij}}$ are the squared norms of the corresponding residual vectors, technically the squared norms of their orthogonal complementary components w.r.t. the inner product. Being shorter, we shall keep on going with the name residual vectors.

\section{Statement of the result}

\begin{theorem}
Let $N\geq 2$, let $X_1,\dots,X_N\in L^2(\Omega,\mathcal F,\mathbb P)$, and fix $i\neq j$. Assume that the principal Gram block $G_{S^{-}_{ij}}$, with $S^{-}_{ij}=\{1,\dots,N\}\setminus\{i,j\}$, is positive definite. Then the following inequality holds
\begin{equation}\label{residual}
\left(
{g}_{ij}-\pi_{ij\mid S^{-}_{ij}}
\right)^2
\leq
\Delta_{i\mid S^{-}_{ij}}\Delta_{j\mid S^{-}_{ij}}\;\;,
\end{equation}
moreover it also holds the estimate
\begin{equation}\label{N_general}
{g}_{ij}^2
\leq
\left(
|\pi_{ij\mid S^{-}_{ij}}|
+
\sqrt{
\Delta_{i\mid S^{-}_{ij}}\Delta_{j\mid S^{-}_{ij}}
}
\right)^2,
\end{equation}
or equivalently,
\begin{equation}
\left(E(X_iX_j)\right)^2
\leq
\left(
|\pi_{ij\mid S^{-}_{ij}}|
+
\sqrt{
\Delta_{i\mid S^{-}_{ij}}\Delta_{j\mid S^{-}_{ij}}
}
\right)^2.
\end{equation}
\end{theorem}
\begin{remark} Since $\Sigma\equiv Cov(X,X)=G(X-E(X))$ the same theorem provides estimates on an off-diagonal entry $\sigma_{ij}$, $i\neq j$, of any covariance matrix in term of all the remaining entries. The same holds for the matrix $R\equiv corr(X,X)=G(Z)$, with $Z$  being the component-wise standardization of $X$.  
\end{remark}
\begin{remark}
(a) For $N=3$, $S^{-}_{ij}$ contains just one index, say $k$. Then $G_{S^{-}_{ij}}=({g}_{kk})$, $g_i=({g}_{ik})$, and $g_j=({g}_{jk})$. Hence
\begin{equation}
\pi_{ij\mid k}=\frac{{g}_{ik}{g}_{jk}}{{g}_{kk}},
\qquad
\Delta_{i\mid k}={g}_{ii}-\frac{{g}_{ik}^2}{{g}_{kk}},
\qquad
\Delta_{j\mid k}={g}_{jj}-\frac{{g}_{jk}^2}{{g}_{kk}}.
\end{equation}
The theorem gives
\begin{equation}\label{N=3}
{g}_{ij}^2g_{kk}^2
\leq
\left(
\left|{g}_{ik}{g}_{jk}
\right|
+
\sqrt{
\left({g}_{ii}g_{kk}-{g}_{ik}^2\right)
\left({g}_{jj}g_{kk}-{g}_{jk}^2\right)
}
\right)^2,
\end{equation}
then, by using the algebraic identity 
$$
\left(
|ab|+
\sqrt{(Ac-a^2)(Bc-b^2)}
\right)^2
=
ABc^2
-
\left(
|a|\sqrt{Bc-b^2}
-
|b|\sqrt{Ac-a^2}
\right)^2,
$$
inequality (\ref{N=3}) exactly reproduces Theorem 3.1 of [LT] which extends Walker's inequality, this last being recovered by choosing one of the $N=3$ variables to be the identity.\\
(b) For $N=4$, if $i=1$ and $j=2$, then $S^{-}_{12}=\{3,4\}$, $G_{S^{-}_{12}}=G_{34}$, $g_1=({g}_{13},{g}_{14})^T$, and $g_2=({g}_{23},{g}_{24})^T$, the inequality (\ref{N_general}) takes the explicit form 
\begin{equation}\label{N=4}
\begin{aligned}
\left| g_{12} \right|
\leq\;&
\left|
\frac{
g_{44}g_{13}g_{23}
-
g_{34}g_{13}g_{24}
-
g_{34}g_{14}g_{23}
+
g_{33}g_{14}g_{24}
}{
g_{33}g_{44}-g_{34}^2
}
\right|
\\[0.5em]
&+
\sqrt{
g_{11}
-
\frac{
g_{44}g_{13}^2
-
2g_{34}g_{13}g_{14}
+
g_{33}g_{14}^2
}{
g_{33}g_{44}-g_{34}^2
}
}
\sqrt{
g_{22}
-
\frac{
g_{44}g_{23}^2
-
2g_{34}g_{23}g_{24}
+
g_{33}g_{24}^2
}{
g_{33}g_{44}-g_{34}^2
}
}.
\end{aligned}
\end{equation}
Estimates for other pairs $(i,j)$ are obtained by the same formula after relabelling the indices. From (\ref{N=4}) we recover (\ref{N=3}) by setting $g_{14}=g_{24}=g_{34}=0$ and multiplying both sides by $|g_{33}|$.
\end{remark}
\begin{remark}
If $G_{S^{-}_{ij}}$ is only positive semidefinite and not invertible, then the auxiliary variables $(X_h)_{h\in S^{-}_{ij}}$ are linearly dependent in $L^2$. In that case one may remove redundant variables, or replace $G_{S^{-}_{ij}}^{-1}$ by the Moore-Penrose pseudo-inverse.
\end{remark}

\begin{proof}
Let $V:=\operatorname{span}\{X_h:h\in S^{-}_{ij}\}$ and let $P$ be the orthogonal projection onto $V$. In the case $S^{-}_{ij}=\varnothing$, we use the convention $V=\{0\}$ and $P=0$. Define the residual vectors $R_i:=X_i-PX_i$ and $R_j:=X_j-PX_j$, then
$$
X_i=PX_i+R_i,
\qquad
X_j=PX_j+R_j,
$$
with $R_i,R_j\perp V$. Therefore
$$
g_{ij}
=
\langle X_i,X_j\rangle
=
\langle PX_i,PX_j\rangle
+
\langle R_i,R_j\rangle .
$$
If $S^{-}_{ij}\neq\varnothing$, the projection coefficients of $X_i$ and $X_j$ onto $V$ are determined by the normal equations
$$
G_{S^{-}_{ij}}c_i=g_i,
\qquad
G_{S^{-}_{ij}}c_j=g_j.
$$
Since $G_{S^{-}_{ij}}$ is positive definite, it is invertible, and therefore
$$
c_i=G_{S^{-}_{ij}}^{-1}g_i,
\qquad
c_j=G_{S^{-}_{ij}}^{-1}g_j.
$$
Hence
$$
\langle PX_i,PX_j\rangle
=
g_i^TG_{S^{-}_{ij}}^{-1}g_j
=
\pi_{ij\mid S^{-}_{ij}}.
$$
Similarly,
$$
\|R_i\|^2
=
\langle X_i-PX_i,X_i-PX_i\rangle
=
g_{ii}-g_i^TG_{S^{-}_{ij}}^{-1}g_i
=
\Delta_{i\mid S^{-}_{ij}},
$$
and
$$
\|R_j\|^2
=
\langle X_j-PX_j,X_j-PX_j\rangle
=
g_{jj}-g_j^TG_{S^{-}_{ij}}^{-1}g_j
=
\Delta_{j\mid S^{-}_{ij}}.
$$
We notice that if $S^{-}_{ij}=\varnothing$, these identities reduce to $\pi_{ij\mid S^{-}_{ij}}=0$, $\Delta_{i\mid S^{-}_{ij}}=g_{ii}$ and $\Delta_{j\mid S^{-}_{ij}}=g_{jj}$.
Thus in general we have
$$
g_{ij}
=
\pi_{ij\mid S^{-}_{ij}}
+
\langle R_i,R_j\rangle .
$$
and applying Cauchy--Schwarz to the residual vectors gives
$$
\left(
g_{ij}-\pi_{ij\mid S^{-}_{ij}}
\right)^2
=
\langle R_i,R_j\rangle^2
\leq
\|R_i\|^2\|R_j\|^2
=
\Delta_{i\mid S^{-}_{ij}}\Delta_{j\mid S^{-}_{ij}},
$$
which proves the inequality (\ref{residual}). Taking square roots, we obtain
$$
|g_{ij}-\pi_{ij\mid S^{-}_{ij}}|
\leq
\sqrt{
\Delta_{i\mid S^{-}_{ij}}\Delta_{j\mid S^{-}_{ij}}
}.
$$
Therefore, by the triangle inequality, we have
$$
|g_{ij}|
=
\left|
\pi_{ij\mid S^{-}_{ij}}
+
g_{ij}-\pi_{ij\mid S^{-}_{ij}}
\right|
\leq
|\pi_{ij\mid S^{-}_{ij}}|
+
|g_{ij}-\pi_{ij\mid S^{-}_{ij}}|.
$$
and consequently
$$
|g_{ij}|
\leq
|\pi_{ij\mid S^{-}_{ij}}|
+
\sqrt{
\Delta_{i\mid S^{-}_{ij}}\Delta_{j\mid S^{-}_{ij}}
}.
$$
Squaring it gives
$$
g_{ij}^2
\leq
\left(
|\pi_{ij\mid S^{-}_{ij}}|
+
\sqrt{
\Delta_{i\mid S^{-}_{ij}}\Delta_{j\mid S^{-}_{ij}}
}
\right)^2.
$$
which proves the  estimate (\ref{N_general}).
\end{proof}
\section{Cramer Rao improved estimates}

We now indicate how the preceding inequality can be used to refine Cramer-Rao (CR) lower bounds for biased estimators. Let $(P_\theta)_{\theta\in\Theta}$ be a regular one-dimensional statistical model with density $f_\theta$, and let $T$ be an estimator of $\theta$. Define the estimation error $A_\theta \equiv T-\theta$, the bias $b(\theta)\equiv E_\theta(A_\theta)=E_\theta(T)-\theta$, and the score $U_\theta \equiv \partial_\theta\log f_\theta$. Under the usual regularity assumptions, $E_\theta(U_\theta)=0$, $E_\theta(U_\theta^2)=I(\theta)$, the Fisher information, and
\begin{equation}\label{derivative}
E_\theta(A_\theta U_\theta)=1+b'(\theta).
\end{equation}
The ordinary CS inequality gives $\left(1+b'(\theta)\right)^2\leq E_\theta(A_\theta^2)I(\theta)$,
hence $E_\theta(A_\theta^2)\geq (1+b'(\theta))^2/I(\theta)$, that is CR inequality for biased estimators. Walker's improvement corresponds to the estimate $\left(1+b'(\theta)\right)^2\leq\left(E_\theta(A_\theta^2)-b(\theta)^2\right)I(\theta)$, see [W], which gives the sharper lower bound
\begin{equation}\label{Walker}
MSE(T)\equiv E_\theta(A_\theta^2)
\geq
b(\theta)^2+
\frac{\left(1+b'(\theta)\right)^2}{I(\theta)}.
\end{equation}
The $N$-variables inequality proven in the previous section may produce further refinements whenever auxiliary information is at disposal. Indeed, let $H_1,\dots,H_m$ be a set of auxiliary square-integrable random variables, and consider the vector $H=(H_1,\dots,H_m)^T$. Assume that the Gram matrix $M_\theta:=E_\theta(HH^T)$ is positive definite and define
\begin{equation}
a_\theta \equiv E_\theta(A_\theta H),
\qquad
u_\theta \equiv E_\theta(U_\theta H).
\end{equation}
Consider the $N=2+m$ random vector $X=(A_\theta,U_\theta,H_1,\ldots,H_m)$ , by applying the residual inequality (\ref{residual}) to the pair $(A_\theta,U_\theta)$ it is obtained
$$
\left(
E_\theta(A_\theta U_\theta)-a_\theta^TM_\theta^{-1}u_\theta
\right)^2
\leq
\left(
E_\theta(A_\theta^2)-a_\theta^TM_\theta^{-1}a_\theta
\right)
\left(
I(\theta)-u_\theta^TM_\theta^{-1}u_\theta
\right).
$$
Recalling  (\ref{derivative}), whenever $I(\theta)-u_\theta^TM_\theta^{-1}u_\theta>0$, this yields the generalized lower bound
\begin{equation}\label{general_CR}
E_\theta(A_\theta^2)
\geq
a_\theta^TM_\theta^{-1}a_\theta
+
\frac{
\left(
1+b'(\theta)-a_\theta^TM_\theta^{-1}u_\theta
\right)^2
}{
I(\theta)-u_\theta^TM_\theta^{-1}u_\theta.
}
\end{equation}
We can compare formula (\ref{general_CR}), specializing it  to $N=3$, with the corresponding formula obtained in [LT] (supplementary material). For this comparison we take $m=1$ in (\ref{general_CR}) and set $H=H_1\equiv Z$. Then, using the same notations as in [LT], we set
$S_Z\equiv M_\theta=E_\theta(Z^2)$, $a\equiv a_\theta=E_\theta(A_\theta Z)$ and $b\equiv u_\theta=E_\theta(U_\theta Z)$.
We also set $c\equiv E_\theta(A_\theta U_\theta)=1+b'(\theta)$ and recall $E_\theta(U_\theta^2)=I(\theta)$. Then (\ref{general_CR}) produces the lower bound
\begin{equation}\label{me}
\operatorname{MSE}(T)=E_\theta(A_\theta^2)
\geq
\frac{a^2}{S_Z}
+
\frac{
\left(
c-\frac{ab}{S_Z}
\right)^2
}{
I(\theta)-\frac{b^2}{S_Z}
}=\frac{a^2}{S_Z}
+
\frac{
\left(
S_Zc-ab
\right)^2
}{
S_Z\left(I(\theta)S_Z-b^2\right)
},
\end{equation}
to be compared with the lower bound in [LT] which is\footnote{the original [LT] (supplementary material) formula contains some typos}
\begin{equation}\label{LT}
\operatorname{MSE}(T)
\geq
\frac{a^2}{S_Z}
+
\frac{
\left(
S_Z|c|-|ab|
\right)_{+}^2
}{
S_Z\left(I(\theta)S_Z-b^2\right)
}.
\end{equation}
Thus the difference between the two formulas is the replacement of the quantity  $(S_Z|c|-|ab|)^2_{+}$ appearing in (\ref{LT}) by the quantity $(S_Zc-ab)^2$. Since
$$
|S_Zc-ab|\geq \bigg||S_zc|-|ab|\bigg|\geq (S_z|c|-|ab|)_{+}
$$
 the first one tends to be sharper (with significant differences  in some cases). Notice that for $Z=1$, being $b=0$, Walker's lower bound is recovered by both formulas.\\
We now show how our lower bound (\ref{general_CR}) allows for interesting generalizations. Consider Walker example of shrinkage estimators $T=w\widehat\theta+(1-w)\mu$, where the estimator $\widehat\theta$ is unbiased and $\mu$ is a reference value for $\theta$ (its prior mean in the Bayesian approach). A direct computation proves that, within this class of biased estimators, Walker's lower bound turns out to be optimal, equalizing the value
\begin{equation}\label{optimal}
MSE(T)=(1-w)^2(\mu-\theta)^2+\frac{w^2}{I(\theta)},
\end{equation}
with $I(\theta)=1/Var_{\theta}(\widehat\theta)$. Notice that explicit computability and optimality  are essentially due to the fact that $T$ is an affine function of $\widehat\theta$ and also an affine function of the score $U_{\theta}$. However such linear dependence in general does not hold for robust biased estimators because of the presence of nonlinear influence functions; in these cases Walker's bound ceases to be optimal.\\
To make the point, consider the elementary model $ X\sim \operatorname{Bin}(m,p),0<p<1,$ where $m$ is known and $p$ is the unknown parameter. For $0<\varepsilon<\frac1{2}$ define the function $t_\varepsilon(x)\equiv \min\left\{1-\varepsilon,\max\left\{\varepsilon,\frac{x}{m}\right\}\right\}$,
then set $T_\varepsilon(X)\equiv t_\varepsilon(X)$. This estimator is statistically meaningful because the unbiased estimator
$T=X/m$ may take the boundary values $0$ and $1$. Such boundary values are often inconvenient if the next step involves log-odds or likelihood ratios. Clipping moves the estimator away from the boundary, but introduces bias. As before, considering $A_{p,\varepsilon}=T_{\varepsilon}-p$ and $b_{\varepsilon}(p)=E_p(T_{\varepsilon})-p$, Walker's lower bound is ($I=\frac{m}{p(1-p)}$):
$$
E_p(A_{p,\varepsilon}^2)
\geq
LB_{Walker}\equiv b_\varepsilon(p)^2
+
\frac{(1+b_\varepsilon'(p))^2}{I}.
$$
We can improve over it by including a natural auxiliary variable $T^2=(\frac{X}{m})^2$ and considering our lower bound (\ref{general_CR}) for $N=4$ and $H=(1,T^2)$, obtaining (after some algebra)
\begin{equation}\label{N=4!}
E_p(A_{p,\varepsilon}^2)\geq LB_{Walker}+
\frac{(I C_\varepsilon-BD)^2}{(I^2v-ID^2)},
\end{equation}
with $v=\operatorname{var}_p(T^2)>0$, $C_{\varepsilon}=\operatorname{cov}_p(A_{p,\varepsilon},T^2)$, $D=\operatorname{cov}_p(U_{p},T^2)$ and $B=1+b'_{\varepsilon}$. Therefore, if $m\geq 2$, $T^2$ is not an affine function of the score $U_p$ (actually it is a quadratic polynomial),  hence $Iv>D^2$ and (\ref{N=4!}) improves over Walker's bound in all cases for which $I C_\varepsilon\neq BD$. Contrary to that, the alternative choice of auxiliary variables given by $H=(1,T)$ produces $Iv=D^2$,  hence it gives rise to degeneracy. Further examples of meaningful biased estimators can be built  for other statistical models showing that our $N=4$ bound on the MSE may produce improvements over Walker or Lupu-Tanase bounds, the above one was chosen because of its simplicity.

\end{document}